\documentclass[11pt]{amsart}
\usepackage{amsmath,amscd,amssymb,amsfonts,amsthm}
\usepackage{tikz-cd} 
\usepackage{tikz}
\usepackage{comment}
\usepackage{yhmath}
\usepackage[color,cmtip,matrix,arrow]{xy}

\newtheorem{theorem}{Theorem}[section]
\newtheorem{corollary}[theorem]{Corollary}
\newtheorem{proposition}[theorem]{Proposition}

\newtheorem{lemma}[theorem]{Lemma}
\theoremstyle{definition}
\newtheorem{definition}[theorem]{Definition}
\theoremstyle{remark}

\newtheorem{remark}[theorem]{Remark}

\newtheorem{example}[theorem]{Example}

\newcommand{\SP}[1]{{\left\langle {{#1}} \right\rangle}}

\newcommand{\fa}{\mathfrak{a}}

\newcommand{\T}{\mathbb{T}}

\newcommand{\ca}{\mathcal}

\newcommand{\E}{\ca{E}}

\newcommand{\R}{\mathbb{R}}

\newcommand\lie[1]{\mathfrak{#1}}

\newcommand{\g}{\lie{g}}

\renewcommand{\a}{\mathsf{a}}

\newcommand{\on}{\operatorname}
\newcommand{\Hom}{ \on{Hom}}

\renewcommand{\subset}{\subseteq}

\newcommand\qu{/\kern-.7ex/} 

\newcommand\lam{\lambda}

\newcommand{\End}{\on{End}}

\begin{document}
	\sloppy
	\title{Local decomposition and linearization of Loday Brackets}
    \author{Hudson Lima}
    \address{UFAM}
    \email{hudsonlima@ufam.edu.br}

\begin{abstract} 
    We study local splitting-type results for general Loday algebroids and use them to obtain a direct proof of the splitting theorem for Courant algebroids. We also discuss the linearization problem and establish a general linearization principle using Euler-like derivations.
\end{abstract}
 
	\maketitle
	\tableofcontents
	
\section{Introduction}
    Lie algebras and Poisson structures are central objects in various areas of mathematics and physics. In differential geometry, a key role is played by geometric versions of Lie algebras known as {\em Lie algebroids}, introduced by Pradines \cite{Pradines} as infinitesimal counterparts of Lie groupoids. Lie algebroids arise in a wide array of contexts and have been the subject of much study; see e.g. \cite{WC99,CFM21,CrMo08,Du01,RLF,Mackenzie05}.

    This paper concerns the study of more general brackets, possibly non-skew-symmetric, known as Loday brackets (after \cite{Loday93}) as defined in \cite{Gra11}. A prominent class of examples is given by {\em Courant algebroids} \cite{LWX97,RoyThesis}, which play a central role in the theory of Dirac and generalized complex structures \cite{Cour90,CoWe88,Gua04}, and have numerous connections with mathematical physics \cite{StrAlek05,RoySigma,Yam07,Zabzine06}. Beyond Courant algebroids, there exist numerous additional examples; see, for instance, \cite{Bar12,CCU20,Gra11,IBMP01,SW08} and the references therein.

    Given a bracket on the space of sections of a vector bundle, one can formalize the notion of local splitting and ask questions about linearization in the same spirit as Weinstein's for Poisson structures \cite{We83}. The case of Lie algebroids is particularly well studied. The {\em Splitting Theorem for Lie algebroids} states that every Lie algebroid is locally the product of a tangent bundle Lie algebroid and a Lie algebroid whose anchor vanishes at the reference point \cite{Dufour2001NormalFF,Du01,RLF,We00}. The natural smooth linearization questions for Lie algebroids have been addressed by Monnier and Zung in \cite{MZ04}. 

    In the case of Courant algebroids, there is an analogous splitting theorem: every Courant algebroid is locally a product of a generalized tangent bundle and a singular Courant algebroid. This result appears in \cite{BBLM}, where the authors construct a general method for obtaining normal forms around transverse submanifolds, which yield splitting theorems in the local case.

    For brackets beyond the Lie algebroid setting, very few linearization results are known, and the problem remains largely open (see also \cite{FeMo04}). The linear part of a singular Courant algebroid can be described as an infinitesimal action of a quadratic Lie algebra with coisotropic isotropies; see \cite{BBLM,MLi}. However, no general linearization results, analogous to those of \cite{MZ04}, are currently available.

    In this paper we study local conditions under which splitting-like theorems hold in the context of almost Loday brackets, with particular focus on Courant brackets. To this end, we establish a link between the tensoriality of the {\em Jacobiator} (its linearity over smooth functions) and the invariance of the anchor and co-anchor under flows generated by sections of the algebroid. We then clarify the notion of direct product of Loday algebroids and introduce the concept of a {\em direct decomposition}. 
    Next, we derive conditions ensuring the existence of local {\em matching decompositions} in which one of the factors is the tangent bundle $TI$ of an open interval $I\subset\R$. As a consequence, we obtain a direct proof of the splitting theorem for Courant algebroids, in the spirit of the original approaches in \cite{Du01,RLF,We00}. We also describe the linear part of singular Loday algebroids and formulate a {\em Linearization Principle} in terms of Euler-like derivations. Finally, we give a simple condition on the representation of the singular Lie algebra that guarantees linearization.

    The paper is organized as follows. In Section~\ref{sec:loday} we review basic properties of Loday algebroids. In Section~\ref{sec:Matching} we define and discuss the types of local product decompositions that will be used throughout the paper. In Section~\ref{sec:localLemmas} we prove the technical lemmas that guarantee the existence of local product decompositions in which one of the factors is $TI$. Section~\ref{sec:Courant} applies these ideas to Courant algebroids and gives a direct proof of the splitting theorem in this case. In Section~\ref{sec:linearization} we study the linearization problem around singular points and establish a linearization principle using Euler-like derivations. Finally, we discuss applications and examples in the Courant algebroid setting.

\section{Loday algebroids} \label{sec:loday}
 An {\em anchored vector bundle} consists of a vector bundle $E\to M$ with a bundle map $\rho\colon E\to TM$ covering the identity on $M$, called {\em anchor}. 
 
 An {\em almost Loday algebroid} (see \cite{Gra11}) over an anchored vector bundle $(E,\rho)$ is a $\R$-bilinear map 
 $$
    \mathfrak{a}\colon \Gamma(E)\times\Gamma(E)\to\Gamma(E)
 $$
 for which there exists a bundle map $\lambda:T^*M\otimes E\otimes E\to E$ satisfying the following identities
	\begin{eqnarray}
		\mathfrak{a}(\alpha,f\beta) &=& f\mathfrak{a}(\alpha,\beta) + (\rho(\alpha)f)\beta \label{eq: leib-anchor}\label{eq:LeibnizRule} \\
		\mathfrak{a}(f \alpha,\beta) &=& f \mathfrak{a}(\alpha,\beta) - (\rho(\beta)f)\alpha + \lambda(df\otimes\alpha\otimes\beta) \label{eq:co-anchorRule}	
	\end{eqnarray}	
	for all $\alpha,\beta \in \Gamma(E)$ and $f\in C^\infty(M)$.

    Since the bracket $\fa$ determines the anchor $\rho$ and co-anchor $\lambda$ uniquely, we will use as convenience  the following notations
    $$
        (E,\fa), \ \ (E,\fa,\rho),\ \ \mbox{or}\ \ (E,\fa,\rho,\lambda)
    $$
	to denote the corresponding almost Loday algebroid.
    
	The {\em Jacobiator} of $\fa$ is the map
	\begin{equation}\label{eq:jac}
		J_\mathfrak{a}(\alpha,\beta,\gamma) := \mathfrak{a}(\mathfrak{a}(\alpha,\beta),\gamma) + \mathfrak{a}(\beta,\mathfrak{a}(\alpha,\gamma)) - \mathfrak{a}(\alpha,\mathfrak{a}(\beta,\gamma)).
	\end{equation}
	If $J_\fa=0$, we say that $(E,\mathfrak{a},\rho)$ is a {\em Loday algebroid}. 
	
	Throughout the paper we will occasionally use the lighter notation
	$$
        [\alpha, \ \beta]_{\fa}=\mathfrak{a}(\alpha, \ \beta)
    $$

    The Jacobiator \eqref{eq:jac} may be interpreted as the Lie derivative of $\fa$ with respect to itself.

    Let $Der(E)$ denote the Lie algebra of derivations of the $C^\infty(M)$-module $\Gamma(E)$. Its elements are pairs $(D,X)$ where $X\in\mathfrak{X}(M)$ and $D\colon\Gamma(E)\to\Gamma(E)$ is an $\R$-linear map satisfying
    $$
        D(f\alpha) = fD(\alpha)+(Xf)\alpha.
    $$
    This condition uniquely determines $X$ called the {\it symbol} of $D$. The Lie bracket on $Der(E)$ is the commutator 
    $$
        \big[D_1, \ D_2\big]_{der} := D_1\circ D_2 - D_2\circ D_1.
    $$
    whose symbol is the Lie bracket of vector fields $[X_1,X_2]$.

Let $\mathfrak{X}_{lin}(E)$ denote the {\em Lie algebra of linear vector fields} on $E$. Each $\widehat{X}\in \mathfrak{X}_{lin}(E)$ defines a derivation 
$$
    D_{\widehat{X}}(\alpha) = \frac{d}{dt}\Big|_{t=0} \big(\varphi_t^{\widehat{X}}\big)^*\alpha,\quad \alpha\in\Gamma(E)
$$
where $\varphi_t^{\widehat{X}}$ denotes the local flow of $\widehat{X}$.

The correspondence $\widehat{X}\mapsto D_{\widehat{X}}$ is a Lie algebra isomorphism
$$
    \mathfrak{X}_{lin}(E)\cong Der(E).
$$

Let $(E,\fa,\rho)$ be an almost Loday algebroid. By (\ref{eq: leib-anchor}), for each $\alpha\in\Gamma(E)$ the map 
$$
    \beta\mapsto [\alpha,\beta]_\fa
$$ 
is a derivation, for every $\alpha \in \Gamma(E)$. Hence the bracket defines a map (also denoted by $\fa$)
$$
    \fa\colon\Gamma(E)\to Der(E)
$$
whose symbol of any $\fa(\alpha)\in Der(E)$ is $\rho(\alpha)$. From this point of view the Jacobiator may be written as 
\begin{equation}\label{eq:jacDerivation}
	J_\fa(\alpha,\beta,-) = \fa\big([\alpha,\beta]_\fa\big) - \big[\fa(\alpha),\fa(\beta)\big]_{der} \ ,
\end{equation}
a derivation with symbol $\rho\big(\!\! \ [\alpha,\beta]_\fa\big)-[\rho(\alpha),\rho(\beta)]$. Idendifying $\fa(\alpha)$ with the corresponding linear vector field on $E$,
\begin{equation}
    J_\fa(\alpha,\beta,\gamma) = \mathcal{L}_{\fa(\alpha)}(\fa)(\beta,\gamma).
\end{equation}
Thus the Jacobi identity means that the flow of each derivation $\fa(\alpha)$ preserves the bracket.

We say that a Loday bracket $\fa$ preserves a tensor on $E$ if the flow of each derivation $\fa(\alpha)$ preserves it.

In particular, the Jacobi identity implies that $\fa$ preserves its anchor and co-anchor. For instance
$$
    \rho([\alpha,\beta]_\fa)=[\rho(\alpha),\rho(\beta)].
$$ 

In fact, this property follows from the tensoriality of $J_\fa$, as the next proposition shows.

\begin{proposition}\label{prop:tensorialLoday}
	Let $\fa$ be a Loday bracket on $E\to M$ with anchor $\rho\colon E\to TM$ and co-anchor $\lambda\colon T^*M\otimes E\otimes E\to E$. Let  
    $$
        S_\fa(\alpha,\beta):= [\alpha,\beta]_\fa+[\beta,\alpha]_\fa
    $$ 
    denote its symmetrization.
    
    The Jacobiator $J_\fa$ is $C^\infty(M)$-linear in each argument if and only if:
	\begin{enumerate}
		\item 
		$$
            \rho([\alpha,\beta]_\fa)=[\rho(\alpha),\rho(\beta)];
        $$
		\item the bracket preserves the co-anchor $\lambda$:
        $$
    		\begin{aligned}
                [\alpha,\lambda(\xi\otimes\beta\otimes\gamma)]_\fa
                &=\lambda(\mathcal{L}_{\rho(\alpha)}\xi\otimes\beta\otimes\gamma)\\
                &\quad+\lambda(\xi\otimes[\alpha,\beta]_\fa\otimes\gamma)\\
                &\quad+\lambda(\xi\otimes\beta\otimes[\alpha,\gamma]_\fa);
            \end{aligned}
        $$
		\item 
		$$\fa(\lambda(df\otimes\alpha\otimes\beta),\gamma)=\rho(\gamma)fS_\fa(\alpha,\beta)-\lambda(df\otimes S_\fa(\alpha,\beta)\otimes\gamma).$$
	\end{enumerate}
\end{proposition}

\begin{remark}
    The tensor $\lam$ is related with $S_\fa$ by
	\begin{equation}
		S_\fa(f\alpha,\beta) = fS_\fa(\alpha,\beta) + \lambda(df\otimes\alpha\otimes\beta),
	\end{equation}
	In particular, if $\rho([\alpha,\beta]_\fa)=[\rho(\alpha),\rho(\beta)]$, then 
    \begin{equation}
        \rho(S_\fa(\alpha,\beta)) = 0, \quad \rho\circ\lambda\equiv 0.
    \end{equation}
\end{remark}

\begin{proof}[Proof of Proposition \ref{prop:tensorialLoday}]
	Item (a) follows from Equation (\ref{eq:jacDerivation}) and from the fact that a derivation is $C^\infty(M)$-linear if and only if its symbol vanishes.
	
	(b) Assume (a). Consider the symmetrization
	
	\begin{equation}
	J_\fa(\alpha,\beta,\gamma)+J_\fa(\alpha,\gamma,\beta)=S_\fa([\alpha,\beta]_\fa,\gamma)+S_\fa(\beta,[\alpha,\gamma]_\fa)-[\alpha,S_\fa(\beta,\gamma)]_\fa,
	\end{equation}
	for all $\alpha,\beta,\gamma\in\Gamma(E)$.
	
	Define 
    $$
        [f,J_\fa](\alpha,\beta,\gamma)=J_\fa(\alpha,f\beta,\gamma)-fJ_\fa(\alpha,\beta,\gamma)
    $$
	and use that 
    $$
        J_\fa(\alpha,\beta,f\gamma)=fJ_\fa(\alpha,\beta,\gamma).
    $$ 

    A direct computation gives
    $$
    \begin{aligned}
        [f,J_\fa](\alpha,\beta,\gamma) &= S_\fa([\alpha,f\beta]_\fa,\gamma)+S_\fa(f\beta,[\alpha,\gamma]_\fa)  - [\alpha,S_\fa(f\beta,\gamma)]_\fa\\
        &\quad -fS_\fa([\alpha,\beta]_\fa,\gamma)-fS_\fa(\beta,[\alpha,\gamma]_\fa)+f[\alpha,S_\fa(\beta,\gamma)]_\fa \\
        &= S_\fa(f[\alpha,\beta]_\fa+\rho(\alpha)f\beta,\gamma)+\lambda(df\otimes\beta\otimes [\alpha,\gamma]_\fa)-[\alpha,fS_\fa(\beta,\gamma)]_\fa \\
		&\quad -[\alpha,\lambda(df\otimes\beta\otimes\gamma)]_\fa-fS_\fa([\alpha,\beta]_\fa,\gamma)+f[\alpha,S_\fa(\beta,\gamma)]_\fa \\
        &= \lambda(df\otimes [\alpha,\beta]_\fa\otimes\gamma)+S_\fa(\rho(\alpha)f\beta,\gamma)+\lambda(df\otimes\beta\otimes [\alpha,\gamma]_\fa) \\
		&\quad -\rho(\alpha)fS_\fa(\beta,\gamma)-[\alpha,\lambda(df\otimes\beta\otimes\gamma)]_\fa.
    \end{aligned}
    $$
    This expression can be rewritten as 
    $$
        \begin{aligned}
            [f,J_\fa](\alpha,\beta,\gamma) & = \lambda(df\otimes [\alpha,\beta]_\fa\otimes\gamma)+\lambda(df\otimes\beta\otimes [\alpha,\gamma]_\fa) \\
            &\quad +\lambda(d(\rho(\alpha)f)\otimes\beta\otimes\gamma) -[\alpha,\lambda(df\otimes\beta\otimes\gamma)]_\fa,
        \end{aligned}
    $$
    which proves item (b).


Using Equation \eqref{eq:jacDerivation} again we obtain
\begin{equation}
J_\fa(\alpha,\beta)+J_\fa(\beta,\alpha)=\fa(S_\fa(\alpha,\beta)),
\end{equation}
for all $\alpha,\beta\in\Gamma(E)$.	

Let 
$$
    [f,J_\fa](\alpha,\beta,\gamma):=J_\fa(f\alpha,\beta,\gamma)-fJ_\fa(\alpha,\beta,\gamma).
$$
Then
$$
\begin{aligned}
    [f,J_\fa](\alpha,\beta,\gamma) &= [S_\fa(f\alpha,\beta),\gamma]_\fa-f[S_\fa(\alpha,\beta),\gamma]_\fa \\
	&= [\lambda(df\otimes\alpha\otimes\beta)+fS_\fa(\alpha,\beta),\gamma]_\fa-f[S_\fa(\alpha,\beta),\gamma]_\fa \\
	&= [\lambda(df\otimes\alpha\otimes\beta),\gamma]_\fa-\rho(\gamma)fS_\fa(\alpha,\beta)+\lambda(df\otimes S_\fa(\alpha,\beta)\otimes\gamma),
\end{aligned}
$$
which yields the identity in (c) and concludes the proof.

\end{proof}

A \emph{Lie algebroid} is a Loday algebroid $(E,\fa,\rho)$ for which the bracket $\fa\colon\Gamma(E)\times\Gamma(E)\to\Gamma(E)$ is skew-symmetric, i.e., the symmetrization $S_\fa$ vanishes. In particular, the co-anchor of a Lie algebroid also vanishes.

There are many examples of Lie algebroids in the literature (e.g. \cite{WC99,Mackenzie05}). We list some relevant ones:
\begin{enumerate}
	\item A Lie algebra $(\g,[\cdot,\cdot])$, viewed as a Lie algebroid over a point $M=\{*\}$.
	\item An involutive distribution $D\subseteq TM$ equipped with the Lie bracket of vector fields.
	\item The cotangent bundle of a Poisson manifold $(M,\pi)$, with bracket
	$$[\alpha,\beta]_{\pi}=\mathcal{L}_{\pi^\sharp(\alpha)}\beta-\mathcal{L}_{\pi^\sharp(\beta)}\alpha-d\pi(\alpha,\beta),$$
	on $\Omega^1(M)$ and anchor $\pi^\sharp:T^*M\to TM$, $\alpha\mapsto \pi(\alpha,-)$.
\end{enumerate}

A \emph{Courant algebroid} (cf.~\cite{LWX97}) is a vector bundle 
$(E,\SP{\cdot,\cdot})$ endowed with a fiberwise nondegenerate symmetric pairing 
(possibly of indefinite signature), together with a Loday algebroid structure $(E,\fa,\rho)$ such that
\begin{enumerate}
	\item the bracket preserves the pairing:
	\begin{equation*}
	\rho(\alpha)\SP{\beta,\gamma}=\SP{[\alpha,\beta]_\a,\gamma}+\SP{\beta,[\alpha,\gamma]_\a}.
	\end{equation*}
	\item the symmetrization of the bracket is given by
	\begin{equation*}
	S_\fa(\alpha,\beta)=\rho^*d\SP{\alpha,\beta},
	\end{equation*}
	for $\alpha,\beta\in\Gamma(E)$, where
    $$\rho^*:T^*M\to E^*\cong E$$ 
    denotes the transpose of the anchor and the identification $E^*\cong E$ is induced by the pairing $\SP{\cdot,\cdot}$.
\end{enumerate}

Comparing with Definition~2.6.1 in \cite{RoyThesis}, Axiom~1 corresponds to the Jacobi identity and Axiom~3 to the Leibniz rule \eqref{eq:LeibnizRule}, while Axioms~4 and~5 correspond to items~(a) and~(b) of the above definition. Moreover, Axiom~2 follows from the remaining axioms and corresponds to item~(a) of Proposition~\ref{prop:tensorialLoday}; see also \cite{Uchino}.


For Courant algebroids the co-anchor map 
$$
    \lambda\colon T^*M\otimes E\otimes E\to E
$$ is given by 
\begin{equation}
	\lambda(\xi\otimes\alpha\otimes\beta) = \SP{\alpha,\beta}\rho^*\xi,
\end{equation}
for all sections $\alpha, \ \beta \in \Gamma(E)$ and $1$-form $\xi\in\Omega^1(M)$.

Important examples of Courant algebroids include:
\begin{itemize}
	\item A Courant algebroid over a point ($M=\{*\}$) is the same as a Lie algebra equipped with an $\mathrm{ad}$-invariant nondegenerate bilinear form, i.e., a \emph{quadratic Lie algebra}.
	\item Given manifold $M$, the {\em standard Courant algebroid} on $M$ is the Courant algebroid structure on the \emph{generalized tangent bundle} 
    \[
        \T M:= TM\oplus T^*M,
    \]  
    defined by
	\begin{enumerate}
		\item the canonical symmetric pairing: $\SP{X+\alpha,Y+\beta}:=i_X\beta+i_Y\alpha$
		\item the {\em Courant-Dorfman bracket}:
		\begin{equation}\label{eq:CourantDorfman}
		[X+\alpha,Y+\beta]:=[X,Y]+\mathcal{L}_X\beta-i_Yd\alpha,
		\end{equation}
		for every $\alpha,\beta\in\Omega^1(M)$ and $X,Y\in\mathfrak{X}(M)$.
		\item the canonical projection (as anchor): $\pi_{TM}\colon\T M\to TM$.
	\end{enumerate}
	\item One can also consider the {\em twisted Courant-Dorfman} bracket, obtained by adding $i_Xi_Y\eta$ in Eq. \eqref{eq:CourantDorfman}, where $\eta\in\Omega^3(M)$. The resulting bracket defines an almost Courant algebroid structure on $\T M$, which satisfies the Jacobi equation if and only if $\eta$ is closed (see e.g., \cite{SevWeinstein02}).
\end{itemize}


\section{Matching and Direct Decompositions}\label{sec:Matching}

A local splitting theorem describes the local structure of a Loday bracket around a point by decomposing it as a “vector bundle product” of two brackets: one transitive and the other singular. The quotation marks in ``vector bundle product'' reflect the fact that, in principle, there are several ways to extend two brackets to the product of their underlying vector bundles obtaining a Loday bracket. 

Given two smooth vector bundles $E_1\to M_1$ and $E_2\to M_2$, their product defines a smooth vector bundle $E_1\times E_2\to M_1\times M_2$. In terms of the canonical projections $\pi_1\colon M_1\times M_2\to M_1$ and $\pi_2\colon M_1\times M_2\to M_2$, this product is isomorphic to the Whitney sum 
$$	
    E_1\times E_2 \stackrel{\sim}{\longrightarrow}\pi_1^*(E_1)\oplus\pi_2^*(E_2).
$$

The natural inclusions $\Gamma(E_1)\to \Gamma(E_1\times E_2)$ and $\Gamma(E_2)\to\Gamma(E_1\times E_2)$ are given explicitly by
$$\hat{\sigma}_1(p,q) = (\sigma_1(p),0_q) \ \mbox{and} \  \hat{\sigma}_2(p,q) = (0_p, \sigma_2(q)), $$
for every $\sigma_1\in\Gamma(E_1)$ and $\sigma_2\in\Gamma(E_2)$. 

Since $E_1\times E_2\cong \pi_1^*(E_1)\oplus \pi_2^*(E_2)$, we obtain
$$
\Gamma(E_1\times E_2) \cong \Gamma(\pi_1^*E_1)\oplus\Gamma(\pi_2^*E_2).
$$
Recall that for any smooth map $\pi\colon N\to M$ 
$$
\Gamma(\pi^*E)\cong C^\infty(N)\otimes_{C^\infty(M)}\Gamma(E).
$$

If $M_2$ is contractible, any vector bundle $E\to M_1\times M_2$ is 
isomorphic to $\pi_1^*(E|_{M_1\times\{q\}})$ for any $q\in M_2$. 
Indeed, the projection $\pi_1\colon M_1\times M_2\to M_1$ is a homotopy 
equivalence with homotopy inverse $i_q(p)=(p,q)$.

Suppose now both $M_1$ and $M_2$ are contractible, and let $E\to M_1\times M_2$ be a vector bundle. Assume that $E=E_1\oplus E_2$. Then any choice of frame for $E_1\to M_1\times M_2$ and $E_2\to M_1\times M_2$ induces an isomorphism
\begin{equation} \label{iso: splits}
    E\cong E_1\big|_{M_1\times\{q\}} \times E_2\big|_{\{p\}\times M_2},
\end{equation}
for every $p\in M_1$ and $q\in M_2$.

The key idea in constructing local splittings is to obtain a local decomposition of $E=E_1\oplus E_2$ where it is possible to find frames making the isomorphism (\ref{iso: splits}) into the desired splitting. 

We will use the following terminology to describe the different types of decompositions that will appear in the next section.

\begin{definition}\label{def:matching}
	Let $\fa$ be an almost Loday algebroid structure on a product vector bundle $E_1 \times E_2$. We say that $\fa$ is a:
	\begin{enumerate}
		\item \emph{Semi-matching (for the factor $E_1$)} if the bracket restricts to $\Gamma(E_1)$, that is,
		\[
		[\Gamma(E_1),\Gamma(E_1)]_{\fa}\subseteq \Gamma(E_1).
		\]
		
		\item \emph{Matching} if the bracket restricts to both $\Gamma(E_1)$ and $\Gamma(E_2)$, i.e.,
		\[
		[\Gamma(E_i),\Gamma(E_i)]_{\fa}\subseteq \Gamma(E_i), \qquad i=1,2.
		\]
		
		\item \emph{Semi-direct decomposition} if, in addition to being a matching, we have
		\[
		[\Gamma(E_1),\Gamma(E_2)]_{\fa}=\{0\}.
		\]
		
		\item \emph{Direct decomposition} if, in addition to being a semi-direct decomposition, we have
		\[
		[\Gamma(E_2),\Gamma(E_1)]_{\fa} = \{0\}.
		\]
	\end{enumerate}
\end{definition}

 Due to the generality of Definition \ref{def:matching}, the Loday algebroid structures in $\Gamma(E_1)$ and $\Gamma(E_2)$ do not determine a unique direct product structure on $\Gamma(E_1\times E_2)$.
 
 A natural question is therefore the following: given two Loday algebroids $(E_1,\fa_1)$ and $(E_2,\fa_2)$, which direct decompositions $(E_1\times E_2,\fa)$ restrict to the original Loday brackets on each factor?

\begin{proposition}
	Let $(E_1, \fa_1)$ and $(E_2,\fa_2)$ be almost Loday algebroids. Then the direct decompositions inducing these given structures are determined uniquely by the choice of
	four bundle maps over $M_1\times M_2$ (which give the mixed components of the co-anchor $\lambda$):
    \[
    \begin{aligned}
        \Lambda^1:\; T^*M_1\otimes E_2\otimes E_1 \longrightarrow E_1\oplus E_2,\\ 
         \Lambda^2:\; T^*M_1\otimes E_2\otimes E_2 \longrightarrow E_1\oplus E_2,\\
         \Lambda^3:\; T^*M_2\otimes E_1\otimes E_1 \longrightarrow E_1\oplus E_2,\\
         \Lambda^4:\; T^*M_2\otimes E_1\otimes E_2 \longrightarrow E_1\oplus E_2.
    \end{aligned}
    \]
\end{proposition}

\begin{proof}
    The bracket is determined on $\Gamma(E_1)\cup\Gamma(E_2)$, which is a set of generators for $\Gamma(E_1\times E_2)$. Hence the bracket is completely determined by the anchor $\rho$ and co-anchor $\lambda$. 

    Let $\rho_1$ and $\rho_2$ denote the anchors of $(E_1,\fa_1)$ and $(E_2,\fa_2)$. We claim that $\rho=\rho_1\times\rho_2$. 
    
    Indeed, let $\sigma_1, \sigma_1'\in\Gamma(E_1)$, $\sigma_2\in\Gamma(E_2)$, $f_1\in C^\infty(M_1)$ and $f_2\in C^\infty(M_2)$. Then
    \[
    \begin{aligned}
        [\sigma_1,f_1\sigma_1']_\fa &= f_1[\sigma_1,\sigma_1'] + (\rho(\sigma_1)f_1) \sigma_1'\\
        0=[\sigma_1,f_2\sigma_2]_\fa &= (\rho(\sigma_1)f_2) \sigma_2\\
    \end{aligned}
    \]
    It follows that $\rho(\sigma_1)f_1\in C^\infty(M_1)$ and $\rho(\sigma_1)f_2=0$, which implies $\rho(\sigma_1)\in\mathfrak{X}(M_1)$. Therefore $\rho(\sigma_1)=\rho_1(\sigma_1)$. An analogous argument shows that $\rho(\sigma_2)=\rho_2(\sigma_2)$.

    Let $\lambda_1$ and $\lambda_2$ denote the co-anchors of $\fa_1$ and $\fa_2$. For $(p,q)\in M_1\times M_2$, the co-anchor $\lambda_{(p,q)}$ of $\fa$ has domain 
    $$
        (T^*_pM_1\oplus T^*_qM_2)\times((E_1)_p\oplus (E_2)_q)\times((E_1)_p\oplus (E_2)_q).
    $$
    Thus $\lambda$ decomposes into eight components. 
    
    Consider the symmetrization 
    $$
        S_\fa(\alpha,\beta)=[\alpha,\beta]_\fa+[\beta,\alpha]_\fa
    $$ 
    Then $S_\fa$ coincides with $S_{\fa_i}$ in $\Gamma(E_i)$ for $i=1,2$ and $S_\fa(\alpha,\beta)=0$ whenever $\alpha\in\Gamma(E_1)$ and $\beta\in\Gamma(E_2)$. 

    We claim that $\lambda$ is determined on all components except those listed in the statement of the proposition. This immediately follows from the identity
    \[
        \lambda(dF,\alpha,\beta) = S_\fa(F\alpha,\beta) - FS_\fa(\alpha,\beta),
    \]
    for all $F\in C^\infty(M_1\times M_2)$ and $\alpha,\beta\in\Gamma(E_1\times E_2)$.
    
\end{proof}

The \emph{direct product} of two Loday algebroids is obtained by choosing the mixed components $\Lambda^j=0$, $j=1,2,3,4$. In this case we denote the co-anchor by $\lambda = \lambda_1\times\lambda_2$.

On the level of sections, the direct product can be characterized by the following condition:
\begin{equation}\label{eq:lie_direct}
	\fa(f\alpha,g\beta) = -\fa(g\beta,f\alpha) = f\rho_1(\alpha)g.\beta - g\rho_2(\beta)f.\alpha,
\end{equation}
for every $\alpha\in\Gamma(E_1)$, $\beta\in\Gamma(E_2)$, and $f,g\in C^\infty(M_1\times M_2)$.

It follows from \eqref{eq:lie_direct} and Proposition \ref{prop:tensorialLoday} that if the brackets on $E_1$ and $E_2$ satisfy the Jacobi identity, then so does their direct product. 

We obtain the following immediate consequences.

\begin{corollary}
	The direct product of two Loday algebroids is again a Loday algebroid. Moreover, if both Jacobiators of the initial algebroids are trilinear over the ring of smooth functions on the base manifold, then the same holds for the Jacobiator of their direct product.
\end{corollary}

\begin{corollary}
	The direct product of Lie algebroids is again a Lie algebroid. The direct product of Courant algebroids is again a Courant algebroid (with the direct product metric).
\end{corollary}

The second corollary follows by verifying that the usual product brackets for Lie and Courant algebroids satisfy \eqref{eq:lie_direct}.

\section{Local Decompositions Lemmas}\label{sec:localLemmas}

In this section we give conditions ensuring the existence of a local decomposition $TI\times E'$ around a point of the base manifold, where $I$ denotes an open interval (see Lemmas~\ref{lemma:2} and~\ref{lemma:matching}).

Let $\Gamma_{\mathrm{loc}}(E)$ be the space of local smooth sections of $E$.

\begin{lemma}\label{lemma:1}
	Let $(E,\fa)$ be an almost Loday algebroid with anchor $\rho$. Let $p\in M$ and suppose $\alpha\in\Gamma_{\mathrm{loc}}(E)$ satisfies
	$\rho(\alpha)_p\neq 0$. Then, for any $v\in E_p$, there exists a section $\beta\in\Gamma_{\mathrm{loc}}(E)$ such that
	\begin{equation}\label{eq: parametricODE}
	\left\{\begin{array}{l}
	[\alpha,\beta]_\fa=0, \\
	\beta_p=v.
	\end{array}\right. 
	\end{equation}
	Furthermore, any other section satisfying these equations that agrees with $\beta$ on a hypersurface transversal to the flow of $\rho(\alpha)$
	through $p$ is equal to $\beta$ in a sufficiently small neighborhood of $p$.
\end{lemma}

\begin{proof}
	Since $\rho(\alpha)_p\neq 0$, we can choose coordinates $(t,x)$ centered at $p\in M$ such that 
	$$
        \rho(\alpha)=\frac{\partial}{\partial t}.
    $$ 
    If $\Sigma$ is a hypersurface transversal to $\rho$ these coordinates  may be chosen so that $(0,x)$ parametrizes $\Sigma$ is a neighborhood of $p$.
    
	Let $\{\beta_1,...,\beta_r\}\subseteq\Gamma_{\mathrm{loc}}(E)$ be a local frame, and let 
	$\Gamma_j^k,\in C^{\infty}_{\mathrm{loc}}(M)$, for $0\leq j,k\leq r$,
	be defined by
	\begin{equation}
	[\alpha,\beta_j]_\fa = \displaystyle\sum_{k=1}^r\Gamma_j^k\beta_k.
	\end{equation}
	Given $\beta\in\Gamma_{\mathrm{loc}}(E)$, we may write
    \begin{equation}
	   \beta = \displaystyle\sum_{k=1}^r b^k\beta_k.
	\end{equation}
    for some $b^k\in C^{\infty}_{\mathrm{loc}}(M)$. Then
	\begin{eqnarray*}
		[\alpha,\beta]_\fa &=& \sum_j(b^j[\alpha,\beta_j]_\fa+\rho(\alpha)b^j\beta_j) \\
		&=& \sum_{j,k}b^j\Gamma_j^k\beta_k+\sum_k\frac{\partial b^k}{\partial t}\beta_k \\
		&=& \sum_k\big(\frac{\partial b^k}{\partial t}+\sum_jb^j\Gamma_j^k\big)\beta_k.
	\end{eqnarray*}
	Therefore Equation \eqref{eq: parametricODE} is equivalent to solving the following parametric ODE:
	\begin{equation}
	   \left\{\begin{array}{l}
	   \displaystyle\frac{\partial b^k}{\partial t}(t,x)+\sum_jb^j(t,x)\Gamma_j^k(t,x)=0 \\
	   b^k(0,0)=v^k
	\end{array}\right. ,
	\end{equation}
	where $v=\sum_kv^k\beta_k(p)$. 
	
	This ODE admits a solution, which is uniquely determined by the choice of the initial values $b^k(0,x)$. 
	
\end{proof}

We say that an almost Loday algebroid $(E,\fa,\rho)$ is \emph{involutive} if 
\[
    \rho([\alpha,\beta]_\fa) = [\rho(\alpha),\rho(\beta)],
\]
for every $\alpha,\beta \in\Gamma(E)$. By Proposition \ref{prop:tensorialLoday}, every Loday algebroid is involutive.

The next lemma provides a condition for the existence of a local semi-matching of the form $TI\times E'$ (for the factor $E'$). 

\begin{lemma}\label{lemma:2}
	Let $(E,\fa)$ be an involutive almost Loday algebroid with anchor $\rho$. Let $p\in M$, and let $\alpha\in\Gamma_{\mathrm{loc}}(E)$, such that
	$$
        \rho(\alpha)\neq 0, \quad J_\fa(\alpha)\equiv 0,
    $$
	on a neighborhood of $p$.
	Then there exists a neighborhood $U$ of $p$ and a Loday algebroid isomorphism
	$$\varphi: (E,\fa,\rho)|_U \to ( TI\times E', \fa',\mathrm{id}\times \rho'),$$
	where $I\subset\mathbb{R}$ is an open interval and $E'$ is a vector bundle over a hypersurface through $p$ transverse to $\rho(\alpha)$, such that $\varphi_*\alpha=\partial/\partial t$ and 
	$$
        \left[\frac{\partial}{\partial t}, \ \Gamma(E')\right]_{\fa'} = \{0\}.
    $$
	Moreover, $TI\times E'$ is a semi-matching for the factor $E'$.
\end{lemma}

\begin{proof}
    We divide the proof into steps.
    
    \emph{Step 1: Choice of coordinates and basis.}
    Choose coordinates $(t,x)$ centered at $p\in M$ such that 
	$$
        \rho(\alpha)=\frac{\partial}{\partial t}.
    $$ 
    Let $\{v_0,v_1,...,v_r\}\subseteq E_p$ be a basis of $E_p$ satisfying
	$$
    \rho(v_i)t\neq 0, \quad i=0,1,...,r,$$ 
    and set $v_0:=\alpha(p)$. 

    \medskip

    \emph{Step 2: Construction of an adapted frame.}
	
	By Lemma \ref{lemma:1} there exists a local frame 
	$\{\tilde{\beta}_0,...,\tilde{\beta}_r\}\subseteq\Gamma_{\mathrm{loc}}(E)$ such that
    \[
        [\alpha,\tilde{\beta}_i]_\fa = 0, \quad (\tilde{\beta}_i)_p=v_i
    \]

    After possibly shrinking the neighborhood of $p$, we may assume that 
	$$
        \rho(\tilde{\beta}_i)_qt\neq 0
    $$ 
    for every $q$.
	
	Define a new frame $\{\beta_0,...,\beta_r\}\subseteq\Gamma_{\mathrm{loc}}(E)$ by
	\begin{equation}
	\beta_i:= \left\{\begin{array}{ll}
	\displaystyle\frac{\tilde{\beta}_0}{\rho(\tilde{\beta}_0)t}, & i=0, \\
	\displaystyle\frac{\tilde{\beta}_i}{\rho(\tilde{\beta}_i)t}-\frac{\tilde{\beta}_0}{\rho(\tilde{\beta}_0)t}, &  1\leq i\leq r.
	\end{array}\right. 
	\end{equation}
    
    \medskip
    \emph{Step 3: Properties of the adapted frame.}
    
	We first show that
    \[
        [\alpha,\beta_i]_\fa=0, \quad \rho(\beta_i)t=\delta_{i0}.
    \]
	Indeed, it suffices to show that if $[\alpha,\beta]_\fa=[\alpha,\gamma]_\fa=0$, then 
    $$[\alpha,h(\rho(\gamma)t)\beta]_\fa=0,$$
    for any $h\in C^{\infty}(\mathbb{R})$. This follows from the Leibniz rule:
    \begin{eqnarray*}
        [\alpha, \ h(\rho(\gamma)t)\beta]_\fa &=& h(\rho(\gamma)t)[\alpha,\beta]_\fa+\rho(\alpha)(h(\rho(\gamma)t))\beta \\
        &=& h'(\rho(\gamma)t)\rho(\alpha)\rho(\gamma)t\beta \\
        &=& ([\rho(\alpha),\rho(\gamma)]t+\rho(\gamma)\rho(\alpha)t)h'(\rho(\gamma)t)\beta \\
        &=& (\rho([\alpha,\gamma]_\fa)t+\rho(\gamma)1)h'(\rho(\gamma)t)\beta \\
        &=& 0.
    \end{eqnarray*}

    Next we show that, for $1\le i,j\le r$, we have
    \[
        [\beta_i,\beta_j]_\fa=\displaystyle\sum_k\Gamma_{ij}^k(x)\beta_k, \quad \rho(\beta_i)=\displaystyle\sum_l\theta_{il}(x)\frac{\partial}{\partial x_l}
    \]
	
	Write 
    $$
        \rho(\beta_i)=\sum_l\theta_{il}(t,x)\frac{\partial}{\partial x_l}+b(t,x)\frac{\partial}{\partial t}.
    $$
    Since $\rho(\beta_i)t=0$, we have $b(t,x)=0$, hence 
    $$
        \rho(\beta_i)=\sum_l\theta_{il}(t,x)\frac{\partial}{\partial x_l}.
    $$
    Using $[\alpha,\beta_i]_\fa=0$ we obtain
    \begin{eqnarray*}
        0 &=& \rho([\alpha,\beta_i]_\fa)\\
        &=& [\frac{\partial}{\partial t},\rho(\beta_i)] \\
        &=& \sum_l\frac{\partial\theta_{il}}{\partial t}(t,x)\frac{\partial}{\partial x_l},
    \end{eqnarray*}
    hence $\theta_{il}(t,x)=\theta_{il}(x)$.
    
    Similarly, write 
    $$
        [\beta_i,\beta_j]_\fa=\sum_{k=1}^r\Gamma_{ij}^k(t,x)\beta_k+c(t,x)\alpha.
    $$ 
    Then
    $$
        c(t,x)=\rho([\beta_i,\beta_j]_\fa)t=[\rho(\beta_i),\rho(\beta_j)]t=0
    $$
    since $\rho(\beta_i)t=0$, for $i\geq 1$. 
    
    Finally, since $J_\fa(\alpha)\equiv 0$, we have
    \begin{eqnarray*}
        0 &=& J_\fa(\alpha,\beta_i,\beta_j) \\
        &=& [[\alpha,\beta_i]_\fa,\beta_j]_\fa+[\beta_i,[\alpha,\beta_j]_\fa]_\fa-[\alpha,[\beta_i,\beta_j]_\fa]_\fa \\   
        &=& -\sum_{k\geq 1}\left(\Gamma_{ij}^k(t,x)[\alpha,\beta_k]_\fa+\rho(\alpha)\Gamma_{ij}^k\beta_k\right) \\
        &=&-\sum_{k\geq 1}\frac{\partial \Gamma_{ij}^k}{\partial t}(t,x)\beta_k.
    \end{eqnarray*}
    Thus $\Gamma_{ij}^k(t,x)=\Gamma_{ij}^k(x)$.
    
    \medskip
	\emph{Step 4: Construction of the semi-matching.}

	Let $M'$ be the set of points with coordinates $(0,x)$. Let $E'$ be the restriction to $M'$ of $\mathrm{span}\{\beta_i\}_{i\geq 1}$. Define  $\varphi$ to be the isomorphism taking the ordered frame 
    $$
        (\alpha,\beta_1,...,\beta_r)
    $$ to 
    $$
        (\partial/\partial t,\beta_1|_{M'},...,\beta_r|_{M'}).
    $$
    The required properties follow immediately.
	
\end{proof}

In general, the bracket induced on $TI\times E'$ does not restrict to $\Gamma(TI)\subset \Gamma(TI\times E')$, and even when it does, the restriction need not coincide with the Lie bracket of vector fields. The next lemma gives conditions guaranteeing this occurs, and furthermore specifies when the resulting decomposition is semi-direct.

\begin{lemma}\label{lemma:matching}
	Let $(E,\fa,\rho, \lambda)$ be an involutive almost Loday algebroid, let $p\in M$ and let $\alpha\in\Gamma_{\mathrm{loc}}(E)$. Assume that on a neighborhood of $p$ we have
	\[
        \rho(\alpha)\neq 0, \qquad J_\fa(\alpha)\equiv 0, \qquad \fa(\alpha,\alpha)=0,
    \]
    and that there exists a closed $1$-form  $\xi\in\Omega^1_{\mathrm{loc}}(M)$ such that
    \[
        \lambda(\xi\otimes\alpha\otimes\alpha)\equiv 0, \qquad \xi(\rho(\alpha))\equiv 1.
    \]
	
	Then, after possibly shrinking the neighborhood of $p$, there exists an almost Loday isomorphism
	\[
        \varphi: (E,\fa,\rho,\lambda)|_U \to (TI\times E', \fa',\mathrm{id}\times\rho',\lambda'),
    \]
    such that 
    \[
    \varphi_*\alpha=\partial/\partial t, \qquad \varphi_*\xi=dt
    \]
	where $I\subset\mathbb{R}$ is an open interval and $E'$ is a vector bundle over a hypersurface through $p$ transverse to $\rho(\alpha)$.

    Moreover, $TI\times E'$ is a matching satisfying
	$$
        \left[\frac{\partial}{\partial t}, \ \Gamma(E')\right]_{\fa'} = \{0\},
    $$
	and the restriction of $\fa'$ to $\Gamma(TI)$ coincides with the Lie bracket of vector fields. 
	
	If in addition we have 
    \[
        \lambda(\xi\otimes\alpha\otimes -)\equiv 0,
    \] 
    then the decomposition $TI\times E'$ is semi-direct.
\end{lemma}

\begin{proof}
	Since $d\xi=0$ and $\xi(\rho(\alpha))\equiv 1$, we can choose coordinates $(t,x)$ such that
	\[
        \rho(\alpha)=\frac{\partial}{\partial t}, \qquad \xi=dt.
    \]
    Applying Lemma~\ref{lemma:2} we get an isomorphism $\varphi$. Thus it remains to verify that $\fa'$ restricts to $\Gamma(TI)$ and this restriction is the Lie bracket of vector fields.
	
	Since $\varphi$ transports $\fa$ to $\fa'$ we have 
    \[
        [\partial/\partial t,\partial/\partial t]_{\fa'}=0,\qquad \lambda'(dt\otimes \partial/\partial t\otimes \partial/\partial t)=0.
    \]
   Hence, for $f,g\in C^\infty(I)$,
	\begin{eqnarray*}
		\Big[ f\cdot\frac{\partial}{\partial t}, g\cdot\frac{\partial}{\partial t} \Big]_{\fa'} &=& 
		 fg  \Big[\frac{\partial}{\partial t},\frac{\partial}{\partial t}\Big]_{\fa'}
		 + gf'\cdot \lambda'\Big(dt\otimes\frac{\partial}{\partial t}\otimes\frac{\partial}{\partial t}\Big)+ 
		 \big(fg'-gf'\big)\frac{\partial}{\partial t} \\ 
		 &=& \big(fg'-gf'\big)\frac{\partial}{\partial t}\\
		 &=&\Big[ f\cdot\frac{\partial}{\partial t}, g\cdot\frac{\partial}{\partial t} \Big],
	\end{eqnarray*}
    which is the usual Lie bracket.
    
	Finally, suppose that $\lambda(\xi\otimes\alpha\otimes\cdot)\equiv 0$. Then 
    \[
        \lambda'(\xi\otimes\partial/\partial t\otimes -)\equiv 0,
    \]
    and therefore
	\[
        \Big[f\cdot\frac{\partial}{\partial t}, \beta\Big]_{\fa'} = f'\cdot\lambda'\Big(dt\otimes\frac{\partial}{\partial t}\otimes\beta\Big)=0,
    \]
	for every $f\in C^\infty(I)$ and $\beta\in\Gamma(E')$. Hence the decomposition is semi-direct.
	
\end{proof}

\section{Splitting for Courant Algebroids}\label{sec:Courant}
In this section we give an alternative proof of the splitting theorem for Courant algebroids from \cite[Corollary 6.1]{BBLM}, through a more elementary approach. We begin with the following lemma.

\begin{lemma}\label{lemma:goodSection}
	Let $(E,\fa,\rho,\SP{\cdot,\cdot})$ be a Courant algebroid and let $p\in M$ with $\rho_p\neq 0$. 
	Then there exists $\alpha\in\Gamma_{\mathrm{loc}}(E)$ such that
    \[
    \rho(\alpha)_p\neq 0, \qquad [\alpha,\alpha]_\fa=0, \qquad \SP{\alpha,\alpha}=0.
    \]
\end{lemma}

\begin{proof}
	Since $[\alpha,\alpha]_\fa=(1/2)\rho^*d\SP{\alpha,\alpha}$, it suffices to find $\alpha$ such that $\rho(\alpha)_p\neq 0$ and
	$\SP{\alpha,\alpha}=0$.
	
	First observe that there exists $v\in E_p$ such that 
    \[
        \rho_p(v)\neq 0, \quad \text{and}\quad \SP{v,v}_p\neq 0.
    \]
    Indeed, otherwise the quadratic function $w\mapsto \SP{w,w}$ would vanish on the nonempty open set 
    \[
        \{w\in E_p \ | \ \rho_p(w)\neq 0\},
    \]
    and hence vanish identically on $E_p$. This would imply that the pairing $\SP{\cdot,\cdot}_p$ is both skew-symmetric and symmetric, and therefore $\SP{\cdot,\cdot}_p\equiv 0$, contradicting the non-degeneracy of the pairing.
	
	Let $\bar{\alpha}\in\Gamma_{\mathrm{loc}}(E)$ be a local section satisfying $\bar{\alpha}_p=v$. Then, after possibly shrinking the neighborhood $U$ of $p$, we have
    \[
        \rho(\bar{\alpha})_q\neq 0\quad\text{and}\quad \SP{\bar{\alpha},\bar{\alpha}}_q\neq 0, \quad \forall q\in U.
    \]
	
	We may further assume that there exist coordinates $(t,x_1,...,x_k)$ on $U$ such that
	\[
	   \rho\!\left(
        \frac{\bar{\alpha}}{\sqrt{|\SP{\bar{\alpha},\bar{\alpha}}|}}
	   \right)
	   =
	   \frac{\partial}{\partial t}.
	\]
	
	Let $Dt:=\rho^*dt$ and define $\alpha$ by
    \[
        \alpha=\frac{\bar{\alpha}}{\sqrt{|\SP{\bar{\alpha},\bar{\alpha}}|}}-\frac{1}{2}\mathrm{sgn}(\SP{\bar{\alpha},\bar{\alpha}})\,Dt,
    \]
	where $\mathrm{sgn}:\mathbb{R}\setminus\{0\}\to\{-1,1\}$ denotes the sign function. 
	
	Since $\rho(Dt)=0$, we have $\rho(\alpha)=\partial/\partial t$. Moreover
	\begin{eqnarray*}
		\SP{\alpha,\alpha} 
        &=& 
        \frac{\SP{\bar{\alpha},\bar{\alpha}}}{|\SP{\bar{\alpha},\bar{\alpha}}|} 
		+ \frac{1}{4}\SP{Dt,Dt} 
        - \mathrm{sgn}(\SP{\bar{\alpha},\bar{\alpha}})
		\SP{\frac{\bar{\alpha}}{\sqrt{|\SP{\bar{\alpha},\bar{\alpha}}|}},Dt}\\
		&=& 
        \mathrm{sgn}(\SP{\bar{\alpha},\bar{\alpha}})
        +\frac{1}{4}\rho(Dt)
        -\mathrm{sgn}(\SP{\bar{\alpha},\bar{\alpha}})
		\rho\!\left(\frac{\bar{\alpha}}{\sqrt{|\SP{\bar{\alpha},\bar{\alpha}}|}}\right)\!t \\
		&=& 
        \mathrm{sgn}(\SP{\bar{\alpha},\bar{\alpha}})
        +0
        -\mathrm{sgn}(\SP{\bar{\alpha},\bar{\alpha}}) 
        =
        0.
	\end{eqnarray*}
\end{proof}

Given a manifold $M$ we denote by $\fa_M$ the bracket defined in \eqref{eq:CourantDorfman} on the generalized tangent bundle $\T M=TM\oplus T^*M$, whose anchor $\rho_M$ is the projection on $TM$.

The following theorem is our main result for Courant algebroids.

\begin{theorem}
	Let $(E,\fa,\rho,\SP{\cdot,\cdot})$ be a Courant algebroid and let $p\in M$ such that $\rho_p\neq 0$.
	Then there exists a neighborhood $U$ of $p$ and a Courant algebroid isomorphism
    \[
        \varphi: \big(E,\fa,\rho, \SP{\cdot,\cdot}\big)|_U \to \Big(\T I\times E', \  \fa_I\times \fa',\ \rho_I\times\rho',\ \SP{\cdot,\cdot}_{\T I}\times \SP{\cdot,\cdot}'\Big),
    \]
	where $I\subset \R$ is an open interval, $E'$ is a vector bundle over a hypersurface through $p$ which is transverse to $\rho$, and	$\fa_I\times\fa'$ denotes the direct product of Courant algebroids.
\end{theorem}

\begin{proof}
	Let $\alpha$ be a section as given is Lemma \ref{lemma:goodSection}. Using Lemma~\ref{lemma:matching}, we can assume $E=TJ\times E''$ as a matching product, that the restriction of the bracket to $\Gamma(TJ)$ is the Lie bracket of vector fields, and that $\alpha=\partial/\partial t$ satisfies 
    \[
    [\alpha,\beta]_\fa=0,\qquad \beta\in\Gamma_{\mathrm{loc}}(E'').
    \]
	
	Observe that $Dt:=(\mathrm{id}\times\rho'')^*dt$ is a section on $\Gamma(E'')$. Indeed, if $(\alpha,\beta_0', ...,\beta_s')$ is a local frame of $TJ\times E''$, we can write
    \[
        Dt = f(x,t)\alpha + \sum_j b_j(x,t) \beta_j',
    \]
	where $f$ and $b_j$ are smooth functions on $M''\times J$. Since $(\mathrm{id}\times\rho'')(Dt)=0$ while 
    \[
        (\mathrm{id}\times\rho'')(Dt)(\alpha)=\partial/\partial t,\qquad (\mathrm{id}\times\rho'')(\beta_j')\in\Gamma(TM'')
    \]
    it follows that $f(x,t)\equiv 0$. Moreover, since $\SP{\alpha,Dt}\equiv 1$, we have 
    \[
        [\alpha,Dt]_\fa = S_\fa(\alpha,Dt) = (1/2)(\mathrm{id}\times\rho'')^*d(1) = 0.
    \]
    On the other hand,
    \[
        [\alpha, Dt]_\fa = \sum_j\frac{\partial b_j}{\partial t}(x,t) \beta_j',
    \]
	which implies $b_j(x,t)=b_j(x)$ and therefore $Dt\in\Gamma(E'')$.
	
	Next choose $v_0 = Dt_p$ and $(v_0,v_1,...,v_s)$ a basis of $E''_{p''}$ (where $p=(0,p'')\in J\times M''$), such that $\SP{\alpha_p,v_j}_p\neq 0$. This can be achieved by a small perturbation of an initial basis in the direction of $Dt_p$. Extend these vectors to obtain a local frame $(\bar{\beta}_0,\bar{\beta}_1,...,\bar{\beta}_s)$ on $\Gamma_{\mathrm{loc}}(E'')$ satisfying 
    \[
        \bar{\beta}_0=Dt,\qquad \SP{\alpha,\bar{\beta}_j}\neq 0\quad (0\leq j\leq s)
    \]
    Let $I\times M'$ be a neighborhood of $p\in J\times M''$ where this frame is defined.
	
	Define $\beta_1,\beta_2,...,\beta_s\in\Gamma_{\mathrm{loc}}(E)$ by
    \[
        \beta_j=\frac{\bar{\beta}_j-\SP{Dt,\bar{\beta}_j}\alpha}{\SP{\alpha,\bar{\beta}_j}}-Dt,
        \qquad
        j=1,2,\ldots,s.
    \]
    Then $(\alpha,Dt,\beta_1,...,\beta_s)$ is a frame of $E|_{I\times M'}$. Let $E'$ be the restriction of $span\{\beta_j\}_{j\geq 1}$ to $M'$ and let $\varphi$ be the isomorphism sending the ordered frame 
    \[
        (  \alpha, Dt, \beta_1,...,\beta_s)
    \]
    to
    \[
        (\partial/\partial t, dt,  \beta_1|_{M'},...,\beta_s|_{M'})
    \]
    on $\T I\times E'$.
	
	We now verify that the Courant algebroid structure induced on $\T I\times E'$ by $\varphi$ is the direct product described in the statement.
    
	\begin{enumerate}
	 	\item $[\alpha,\beta_j]_\fa=0$, since $\beta_j$ is a sum of terms of the form $f(\SP{\gamma_1,\gamma_2})\gamma_3$ with $[\alpha,\gamma_i]_\fa=0$ and $f:J\to\mathbb{R}$. Indeed,
	 	\begin{eqnarray*}
	 		[\alpha,f(\SP{\gamma_1,\gamma_2})\gamma_3]_\fa &=& \frac{\partial f(\SP{\gamma_1,\gamma_2})}{\partial t}\gamma_3 \\
	 		&=& f'(\SP{\gamma_1,\gamma_2})\frac{\partial \SP{\gamma_1,\gamma_2}}{\partial t}\gamma_3 \\
	 		&=& f'(\SP{\gamma_1,\gamma_2})(\SP{[\alpha,\gamma_1]_\fa,\gamma_2}+\SP{\gamma_1,[\alpha,\gamma_2]_\fa})\gamma_3 \\
	 		&=& 0.
	 	\end{eqnarray*}
 	
	 	\item $\SP{\alpha,\beta_j}=0$, since 
        \[
            \SP{\alpha,\beta_j}=\frac{\SP{\alpha,\bar{\beta}_j}}{\SP{\alpha,\bar{\beta}_j}}-\SP{\alpha,Dt}=1-1=0.
        \]
	 	
	 	\item 
	 	$\SP{Dt,\beta_j}=0$, since
        \[
            \SP{Dt,\beta_j}=\rho(\beta_j)t=\frac{\rho(\bar{\beta}_j)t-\SP{Dt,\bar{\beta}_j}}{\SP{\alpha,\bar{\beta}_j}}=0.
        \]
	 	
	 	\item $\frac{\partial}{\partial t}\SP{\beta_i,\beta_j}=0$, since 
        \[
        \frac{\partial}{\partial t}\SP{\beta_i,\beta_j}=\rho(\alpha)\SP{\beta_i,\beta_j}
	 	=\SP{[\alpha,\beta_i]_\fa,\beta_j}+\SP{\beta_i,[\alpha,\beta_j]_\fa}=0.
        \]
	 	
	 	\item $[\beta_i,\beta_j]_\fa=\sum_kg_{ij}^k(x)\beta_k$ and $\rho(\beta_j)=\sum_jh_{ij}(x)\frac{\partial}{\partial x_j}$.
	 	Indeed, the same argument used in Step 3 of the proof of Lemma~\ref{lemma:2} guarantees that 
	 	$\rho(\beta_j)=\sum_jh_{ij}(x)\frac{\partial}{\partial x_j}$, and furthermore that 
        \[
            [\beta_i,\beta_j]_\fa=\sum_kg_{ij}^k(x)\beta_k+a(x)Dt.
        \]
	 	Hence $a=\SP{\alpha,[\beta_i,\beta_j]_\fa}=\rho(\beta_i)\SP{\alpha,\beta_j}-\SP{\beta_j,[\alpha,\beta_i]_\fa}=0$.
	\end{enumerate}
	
	Since $\SP{\alpha,\alpha}=\SP{Dt,Dt}=0$ and $\SP{\alpha,Dt}=1$, together with $(b)-(d)$, we obtain 
    \[
        \varphi_*\SP{\cdot,\cdot} = \SP{\cdot,\cdot}_{\T I}\times \SP{\cdot,\cdot}'.
    \]
 
	Moreover, $\rho(\alpha)=\partial/\partial t$, $\rho(Dt)=0$, and $(e)$ gives 
    \[
        \varphi_*(\rho) = \rho_I\times\rho'.
    \]

	Finally, for $f,g\in C^\infty(I\times M')$, and $u,v\in\Gamma(E|_{I\times M'})$, we have
	\begin{equation}\label{eq:order1}
		[fu,gv]_\fa = fg\cdot [u,v]_\fa + f\rho(u)g\cdot v-g\rho(v)f\cdot u+g\SP{u,v}\cdot \rho^*df.
	\end{equation}
	Using $[\alpha,\alpha]_\fa=[\alpha,Dt]_\fa=[Dt,\alpha]_\fa=[Dt,Dt]_\fa=0$ and $(e)$, $\varphi_*\fa$ restricts to $\Gamma(\T I)$ and $\Gamma(E')$, hence defines a matching product. Combining (\ref{eq:order1}) with $(a)-(c)$ shows that $\varphi_*\fa$ is the desired direct product.
	
\end{proof}

Given $I_1, I_2,\ldots, I_n\subset \R$ open intervals, the canonical map 
\[
    \varphi_{can}\colon\T I_1\times \T I_2\times\cdots\times \T I_n \to \T (I_1\times I_2\times\cdots\times I_n)
\]
is a Courant algebroid isomorphism. Combining this observation with the associativity of the direct product for Courant algebroids yields the following corollary.

\begin{corollary}[Splitting Theorem for Courant Algebroids]
	Let $(E,\fa,\rho,\SP{\cdot,\cdot})$ be a Courant algebroid and let $p\in M$ such that 
    \[
        r = \mathrm{rank}(\rho_p)\geq 1.
    \]
	Then there exists a neighborhood $U$ of $p$ and a Courant algebroid isomorphism
	$$\varphi: \big(E,\fa,\rho,\SP{\cdot,\cdot}\big)|_U \to \Big(\T L\times E' , \  \fa_L\times \fa',\ \rho_L\times\rho',\ \SP{\cdot,\cdot}_{\T L}\times \SP{\cdot,\cdot}'\Big),$$
	where $L$ is the product of $r$ open intervals, $E'$ is a vector bundle over a submanifold of codimension $2r$ through $p$ that is transverse to $\rho$, and	$\fa_L\times\fa'$ denotes the direct product of Courant algebroids.
\end{corollary}

\section{Linearization of Loday algebroids}\label{sec:linearization}

In this section we discuss the linearization problem for Loday algebroids. Our approach is inspired by that of \cite{CF11} for Poisson structures (see also \cite[Section 9.5]{CFM21}). We start by observing that the linearization problem is well posed for Loday algebroids. Next, we describe how the Moser path method can be formulated in this setting. Finally, we state a linearization principle based on Euler-like derivations.

We say that $p\in M$ is a \emph{singular point} for the Loday algebroid $(E,\fa,\rho,\lambda)$ if 
\[
    \rho_p\equiv 0,\qquad \lambda_p\equiv 0.
\]

Let $(x_i)$ be a system of local coordinates on $M$ and $(e_k)_k\subset\Gamma_{\mathrm{loc}}(E)$ a local frame around such a singularity. Then there exists functions 
\[
    \lambda_{ijk}^l,\; \theta_{ij}, \; \Gamma_{ij}^k\in C^{\infty}_{loc}(M)
\]
such that 
\begin{equation}\label{eq:functions}
	\begin{array}{l}
		[e_i,e_j]_\fa = \displaystyle\sum_k\Gamma_{ij}^k(x)e_k, \\
		\rho(e_i) = \displaystyle\sum_j\theta_{ij}(x)\frac{\partial}{\partial x_j}, \mbox{ and }\\
		\lambda(dx_i\otimes e_j\otimes e_k) = \displaystyle\sum_l \lambda_{ijk}^l(x)e_l.
	\end{array}
\end{equation}

We define linear maps 
\begin{equation*}
	\begin{array}{l}
		\fa^L:E_p\otimes E_p\to E_p, \\
		\rho^L:E_p\to \End(T_pM), \\
		\lambda^L: T_pM^*\otimes E_p\otimes E_p\to \Hom(T_pM;E_p),
	\end{array}
\end{equation*}
by
\begin{equation}\label{eq: linearfunc}
	\begin{array}{l}
		\fa^L((e_i)_p,(e_j)_p) = \displaystyle\sum_k\Gamma_{ij}^k(p)(e_k)_p, \\[6pt]
		\rho^L((e_i)_p)\,x = \displaystyle\sum_j\left(\sum_k\frac{\partial \theta_{ij}}{\partial x_k}(p)\,x_k\right)\left.\frac{\partial}{\partial x_j}\right|_p, \\[8pt]
		\lambda^L((dx_i)_p, (e_j)_p, (e_k)_p)\,x = \displaystyle\sum_l\left( \sum_m\frac{\partial\lambda_{ijk}^l}{\partial x_m}(p)\, x_m \right)(e_l)_p,
	\end{array}
\end{equation}
for $x=\sum_j x_j(e_j)_p\in T_pM$. 

These maps are independent of the choice of local coordinates and local frame around $p$. 

Finally, we extend $\fa^L$ to the sections of the trivial bundle
\[
    E^L:=T_pM\times E_p\to T_pM
\]
using $\rho^L$ and $\lambda^L$ as anchor and co-anchor, respectively. In this way we obtain a Loday structure 
\[
    (E^L, \fa^L, \rho^L, \lambda^L),   
\]
which we call the \emph{linearization} of $(E,\fa,\rho,\lambda)$ at $p$. 

The linearization problem for Loday algebroids asks whether a Loday algebroid is locally isomorphic to its linearization at a singular point, and under which conditions this occurs. This is a delicate problem (see \cite{MZ04}, where a linearization for Lie algebroids is proved using a Nash-Moser argument).

By an \emph{Euler-like vector field} at $p\in M$, we mean a vector field $X$ vanishing at $p$ whose linearization 
\[
    X^L\colon T_pM\to T_pM
\]
is the identity map (thus corresponding to the Euler vector field on $T_pM$). In local coordinates, if 
\[
    X=\sum_jf_j(x)\partial/\partial x_j
\]
then 
\[
    (X^L)x =\sum_j\left(\sum_k\frac{\partial f_j}{\partial x_k}(p) \,x_k\!\right) \left.\frac{\partial}{\partial x_j}\right|_p.
\]

Euler-like vector fields canonically define a germ of diffeomorphism $\varphi$ from $T_pM$ to $M$ such that
\[
    \varphi(0)=p,\quad \varphi_*(\mathcal{E})=X,
\]
where $\mathcal{E}$ denotes the Euler vector field on $T_pM$. 

More generally, one can consider vector fields vanishing on a submanifold whose normal derivative is the Euler vector field on the normal bundle. In this case, one canonically obtains a germ of tubular neighborhood embedding that takes the Euler vector field to $X$ (see \cite{BBLM,SadeghHigson}). 

Our interest here is in the linear vector fields on a vector bundle $E$ which are Euler-like with respect to a fiber $E_p\subset E$. Such linear vector fields can be described in terms of their corresponding derivations as follows. 

Let $E\to M$ be a vector bundle and let $p\in M$. We say that a derivation $(D,X)$ on $E$ is \emph{an Euler-like derivation at $p$} if
\begin{enumerate}
    \item the symbol $X$ is Euler-like at $p$;
    \item the map $D^L\colon E_p\to E_p$ defined below vanishes. On a local frame $(e_i)$, if 
    \[
        De_i=\sum_jD_i^j(x)e_j
    \]
    then
    \[
        D^L((e_i)_p) = \sum_j D_i^j(p)(e_j)_p.
    \]
\end{enumerate}
It is straightforward to verify that the map $D^L$ does not depend on the choice of local frame.

\begin{proposition}\label{prop:EulerLikeDerivation}
    Let $(D,X)$ be an Euler-like derivation at $p$. Then there exists a canonical local flat connection $\nabla$ such that $D=\nabla_X$ on a neighborhood of $p$.
\end{proposition}

\begin{proof}
    The general theory in \cite{BLM} establishes a bijection between germs of Euler-like derivations and germs of vector bundle isomorphisms between
    \[
        E\to M \quad \text{and} \quad T_pM\times E_p\to T_pM
    \]
    that take $p$ to $0\in T_pM$ and fix the fiber $E_p$ (using the identification $\nu(E,E_p)\cong T_pM\times E_p$). 
    
    If $\nabla^L$ is the canonical flat connection on $T_pM\times E_p\to T_pM$, then the Euler-like derivation $D$ is taken to $\nabla^L_{\E}$, where $\E$ is the Euler vector field on $T_pM$ (since this is the derivation that corresponds to the Euler vector field of $T_pM\times E_p\to T_pM$).
    
\end{proof}

\begin{remark}
    The above result can also be obtained by adapting the argument in the proof of Lemma~\ref{lemma:1} to find a local frame $(\beta_1,\ldots,\beta_r)$ such that $D\beta_i=0$. One may then take $\nabla$ to be the flat connection for which this frame consists of parallel sections.
\end{remark}

Let $(D,X)$ be Euler-like at $p$ and let $(\Phi^D, \phi^D)$ be the local isomorphism from 
\[
    E\to M\quad\text{to}\quad T_pM\times E_p\to T_pM
\]
satisfying
\[
    (\Phi^D)_*D=\nabla^L_{\E}
\]
where $\nabla^L$ is the canonical flat connection on $T_pM\times E_p\to T_pM$ and $\mathcal{E}$ denotes the Euler vector field on $T_pM$. Observe that the existence (and canonicity) of the isomorphism $\Phi^D$ follows from Proposition~\ref{prop:EulerLikeDerivation}.

It follows that, on a neighborhood of $p$, the flow of $D$ is transformed into parallel transport along the flow of $\E$. The latter consists, after reparametrization, of the \emph{zooming out} bundle isomorphisms 
\begin{equation}\label{eq:kappa_t}
	\widehat{\kappa}_t\colon T_pM\times E_p\to  T_pM\times E_p,
    \quad
    (v,w)\mapsto (tv, w), \quad t\in (0,1].
\end{equation}

Let $(E,\fa)$ be a Loday algebroid and let 
\[
\fa_1:=(\Phi^D)_*\fa
\]
 be the corresponding almost Loday algebroid on $T_pM\times E_p$ (restricted to a neighborhood of $0$). 
 
 The linearization problem for $\fa$ around $p$ is therefore equivalent to the linearization of $\fa_1$ around $0\in T_pM$. 

 The following proposition sets up the Moser path trick.

 \begin{proposition}\label{prop:Moser}
     Assume $(E,\fa)$ is a Loday algebroid singular at $p$. Let $D$ be an Euler-like derivation for $E$ at $p$ inducing the zooming-out isomorphisms \eqref{eq:kappa_t}. Then the family of Loday brackets
     \begin{equation}\label{def:deformation}
         \fa_t= (\widehat{\kappa}_t)^*\fa_1, \qquad t\in (0,1],
     \end{equation} 
     extends smoothly to $t=0$ by defining $\fa_0$ to be the linearization of $\fa$ at $p$. Furthermore,
     \[
        \dot{\fa}_t = \frac{1}{t}\,\widehat{\kappa}_t^*(\dot{\fa}_1)
     \]
 \end{proposition}
 \begin{proof}
    Denote $\rho_t$ and $\lambda_t$ the anchor and co-anchor of $\fa_t$. Using \eqref{eq:functions} we obtain 
    \[
	\begin{array}{l}
		[e_i,e_j]_{\fa_t} = \displaystyle\sum_k\Gamma_{ij}^k(tx)e_k, \\
		\rho_t(e_i) = \displaystyle\sum_j\frac{\theta_{ij}(tx)}{t}\frac{\partial}{\partial x_j} \\
		\lambda_t(dx_i\otimes e_j\otimes e_k) = \displaystyle\sum_l \frac{\lambda_{ijk}^l(tx)}{t}e_l,
	\end{array}
    \]
    The smooth extension at $t=0$ follows from the Taylor expansion of the structure functions around $p$, which yields precisely the linearized structure. The formula for $\dot{\fa}_t$ follows by differentiating the above formulas and comparing with its value at $t=1$.
    
 \end{proof}

Moser's method continues as follows. One tries to find a derivation (possibly time-dependent) $\widetilde{D}$ whose flow is defined at $t=1$ and such that the map
\[
    t\in[0,1]\mapsto (\phi_{\widetilde{D}}^t)^*\fa_t,
\]
is constant. This condition being equivalent to 
\[
    \dot{\fa}_t + \mathcal{L}_{\widetilde{D}}\fa_t =0.
\]

Suppose we are given a derivation $D$ satisfying
\begin{equation}\label{eq:triviality}
    \dot{\fa}_1 = \mathcal{L}_{D}\fa.
\end{equation}
Then Proposition~\ref{prop:Moser} gives 
\[
    t\dot{\fa}_t= \mathcal{L}_{\widehat{\kappa}_t^*(D)}\fa_t.
\]
Thus we can linearize $\fa$ if 
\[
\widetilde{D}=-\frac{1}{t}\widehat{\kappa}_t^*(D)
\]
extends smoothly to $t=0$. It is not hard to see that this occurs exactly when $(D,X)$ satisfy
\begin{enumerate}
    \item $D^L$ vanishes;
    \item $X$ vanishes to the second order at $p$.
\end{enumerate}

Ideed, this is the approach taken by Crainic and Fernandes \cite{CF11} to recover Conn's linearization theorem for Poisson structures. 

It would be interesting to frame the above construction in terms of derformation spaces (in the case for transverse submanifolds in \cite{BBLM}).

In what follows we take a simplified approach where we try to find an Euler-like derivation $D$ for which $\fa_1$ is already linear. In this case the family $\fa_t$ is constant and equal $\fa^L$. Since $\widehat{\kappa}_t$ is a reparametrization of the flow of $\nabla^L_{\E}$ this constancy can be rewritten using Lie derivative as $\mathcal{L}_{\nabla^L_{\E}}\fa_1 = 0$, pulling back to $E$ we get
\begin{equation}\label{eq:lda}
	\mathcal{L}_{D}\fa = 0.
\end{equation}

Conversely, if an Euler-like derivation $D$ satisfies \eqref{eq:lda}, then $\fa_t$ must be constant. Hence $\fa_1$ is linear, and therefore $\fa$ is linearizable. This yields the following result.

\begin{theorem}[Linearization Principle]\label{thrm:principle}
	A Loday algebroid $(E,\fa,\rho,\lambda)$ is linearizable around a singular point $p$ if and only if there exists an Euler-like derivation $D$ at $p$ whose flow acts by local automorphisms of $\fa$, that is, 
    \[
        \mathcal{L}_{D}\fa = 0.
    \]
\end{theorem}

Since the bracket $\fa$ satisfies the Jacobi identity, each section $\sigma\in\Gamma(E)$ defines a derivation $\fa(\sigma)$ whose flow acts by local automorphisms of $\fa$. The following result characterizes when such a derivation is Euler-like.

\begin{proposition}
	Let $(E,\fa)$ be a Loday algebroid and let $p\in M$ be a singular point. Given a section $\sigma\in\Gamma(E)$, let $v=\sigma_p$. Then $D=\fa(\sigma)$ is an Euler-like derivation if and only if
	\begin{enumerate}
		\item $\fa^L(v,-)\colon E_p\to E_p$ vanishes identically; and
		\item $\rho^L(v) = \mathrm{id}_{T_pM}$,
	\end{enumerate}
	where $\fa^L$ and $\rho^L$ are the linearizations defined in \eqref{eq: linearfunc}.
\end{proposition}
\begin{proof}
    We have that $D=\fa(\sigma)$. The derivation $D$ is Euler-like if and only if its symbol $\rho(\sigma)$ is Euler-like at $p$ and $D^L=0$.

    The first condition is equivalent to requiring that the linearization 
    \[
        \rho(\sigma)^L = \rho^L(v)
    \]
    corresponds to the Euler vector field on $T_pM$, that is, $\rho^L(v)=\mathrm{id}_{T_pM}$.

    The second condition means that $(D\theta)_p$ vanishes for every $\theta\in\Gamma(E)$. Using 
    \[
        \fa^L(v,\theta_p) = ([\sigma,\theta]_\fa)_p = (D\theta)_p.
    \]
    we see that this is equivalent to $\fa^L(v,-)=0$. The result follows.
    
\end{proof}

Observe that $\fa^L$ gives an algebra structure on $E_p$. Since $J_\fa\equiv 0$, the space $E_p$ becomes a Leibniz algebra and $\rho^L\colon E_p\to\End(T_pM)$ is a representation of $(E_p,\fa^L)$.

In the cases of Lie and Courant algebroids, it is well known that $(E_p,\fa^L)$ is a Lie algebra. Therefore we obtain the following corollary.

\begin{corollary}
	Let $(E,\fa, \rho)$ be a Lie or Courant algebroid whose anchor vanishes at $p\in M$. Let $\g=(E_p,\fa^L)$ and $\rho^L\colon\g\to End(T_pM)$. If there exists $v$ in the center of $\g$ such that 
    \[
        \rho^L(v)=\mathrm{id}_{T_pM}
    \]
    then $(E,\fa,\rho)$ is linearizable around $p$. 
\end{corollary}

In the Courant algebroid case, if $D$ flows by isometries, that is, 
\[
    X_D\SP{\alpha,\beta}=\SP{D\alpha,\beta}+\SP{\alpha,D\beta},
\]
for all $\alpha,\beta\in\Gamma(E)$, then the linearizing map is also an isometry.

Indeed, such Euler-like derivations correspond to germs of isometric trivializations from $(E,\SP{\cdot,\cdot})$ to $(T_pM\times(E_p,\SP{\cdot,\cdot}_p))$. 

Since each $\fa(\sigma)$ flows by isometry, the above corollary yields a linearization that also preserves the metric. This is a very special case of linearization by isometry. In general, there exist Courant algebroids whose bracket is linear but no isometry can be a Loday map. 

We finish with an example that shows that even when the bracket is linearizable, the Courant algebroid need not admit a linearizing isometry (using the constant metric on the linear model).

\begin{example}
The example arises from bundles of quadratic Lie algebras, which are Courant algebroids where the anchor vanishes identically. 

Let $(\g, [\cdot,\cdot]_\g,\SP{\cdot,\cdot}_\g)$ be semisimple Lie algebra of compact type with Killing form $\SP{\cdot,\cdot}_\g$. Taking $E=M\times \g$, we have $\Gamma(E)\cong C^\infty(M,\g)$. Given a smooth function $f\colon M\to\R$, consider the bracket  
\[
    [\alpha,\beta]_{\fa_f} = f[\alpha,\beta]_\g, \ \ \forall \alpha, \beta\in\Gamma(E).
\]
It follows that $(E,\fa_f,\SP{\cdot,\cdot}_\g)$ is a Courant algebroid. 

If we take $M=\R^n$ and consider $\widehat{\kappa}_t$ as in \eqref{eq:kappa_t}, each $\widehat{\kappa}_t$ is a bundle isometry and
\[
    \big[\alpha,\beta\big]_{\widehat{\kappa}_t^*\fa_f } (q) = f(tq)[\alpha(q),\beta(q)]_\g.
\]
Hence the linearization of $\fa_f$ at $p=0$ is 
\[
    \fa_f^L=\fa_{f(0)},
\]
where $f(0)$ denotes the constant function $q\mapsto f(0)$. 

Any Courant algebroid isomorphism 
\[
    (E|_U,\fa_f,\SP{\cdot,\cdot})\to (E|_V,\fa_{f(0)},\SP{\cdot,\cdot})
\]
 for neighborhoods $U$ and $V$ of $0\in \R^n$, induces of a smooth family of isometries 
 \[
    \Phi_q\colon\g\to\g, \qquad q\in U,
 \]
 satisfying 
\begin{equation}\label{eq:isoLin}
\left\{\begin{array}{l}
	f(q)\Phi_q([u,v]_\g)=f(0)[\Phi_q(u),\Phi_q(v)]_\g, \mbox{ and } \\
	 \\
	\SP{\Phi_q(u),\Phi_q(v)}_\g=\SP{u,v}_\g,
\end{array}\right.
\end{equation}
for any $u,v\in\g$ and $q\in U$. 

Taking norms in the first equation of \eqref{eq:isoLin}, then the supremum over unit vectors $u,v\in\g$, and using that $\Phi_q$ is an isometry, we obtain
\[
    |f(q)|\sup_{|u|,|v|=1}\big|[u,v]_\g\big| = |f(0)|\sup_{|u|,|v|=1}\big|[u,v]_\g\big|.
\]

Since $\g$ is never abelian, we must have $|f(q)|=|f(0)|$, for every $q\in U$. Therefore, such an isometry cannot exist if the function $f\colon \R^n\to \R$ is not constant.
\end{example}

\end{document}